\documentclass{amsart}

\usepackage[utf8]{inputenc}
\usepackage{amsmath}
\usepackage{amsthm}
\usepackage{breqn}
\usepackage{tikz}
\usetikzlibrary{chains,fit,shapes}
\usepackage{algorithm}
\usepackage{algorithmic}
\usepackage{comment}
\usepackage{amssymb}
\usepackage{bm}
\usepackage{braket}
\usepackage{theoremref}

\title[Complexity and Ramsey largeness]
{Complexity and Ramsey largeness of sets of oracles separating complexity classes}

\author{Alex Creiner}
\author{Stephen Jackson}

\theoremstyle{theorem}
\newtheorem{theorem}{Theorem}[section]

\newtheorem{lemma}[theorem]{Lemma}
\newtheorem{corollary}[theorem]{Corollary}

\theoremstyle{definition}
\newtheorem{definition}[theorem]{Definition}

\newtheorem{question}[theorem]{Question}

\theoremstyle{plain}

\newcommand{\sP}{\mathcal{P}}
\newcommand{\ac}{\mathsf{AC}}
\newcommand{\bS}{\boldsymbol{\Sigma}}
\newcommand{\bP}{\boldsymbol{\Pi}}
\newcommand{\bD}{\boldsymbol{\Delta}}

\newcommand{\np}{\bm{NP}}
\newcommand{\cnp}{\bm{co}\text{-}\bm{NP}}
\newcommand{\bqp}{\bm{BQP}}
\newcommand{\ps}{\bm{PSPACE}}
\newcommand{\ph}{\bm{PH}}
\newcommand{\ww}{\omega^\omega}
\newcommand{\sm}{\setminus}
\newcommand{\conc}{{}^\smallfrown}
\newcommand{\res}{\restriction}
\newcommand{\ad}{\mathsf{AD}}
\newcommand{\co}{co{\text-}}

\newcommand{\N}{\mathbb{N}}

\begin{document}

\begin{abstract}
We prove two sets of results concerning computational complexity classes. 
The first concerns a variation of the random oracle hypothesis 
posed by Bennett and Gill after they showed that relative to a randomly chosen 
oracle, $\bm{P}\neq \bm{NP}$ with probability $1$. 
This hypothesis was quickly disproven in several ways, most famously in 1992
with the result that $\bm{IP} = \bm{PSPACE}$, in spite of the
classes being shown unequal with probability $1$. Here we propose 
a variation of what it means to be ``large'' using the Ellentuck topology.
In this new context, we demonstrate that the set of oracles separating
$\bm{NP}$ and $\bm{coNP}$ is not small, and obtain similar
results for the separation of $\bm{PSPACE}$ from $\bm{PH}$
along with the separation of $\bm{NP}$ from $\bm{BQP}$.
We demonstrate that this version of the hypothesis turns it into a
sufficient condition for unrelativized relationships, at least
in the three cases considered here. Second, we example the descriptive 
complexity of the classes of oracles providing the separations for 
these various classes, and determine their exact placement in the 
Borel hierarchy. 
\end{abstract}

\maketitle

\section{Introduction}

In 1975, a paper by Baker, Gill and Solovay \cite{solovayOriginal} demonstrated the subtlety of the $\bm{P}$ versus $\bm{NP}$ question by demonstrating the existence of oracles relative to which $\bm{P} \neq \bm{NP}$, as well as oracles relative to which $\bm{P} = \bm{NP}$. Subsequently Bennett and Gill in 1981 \cite{Bennett1981RelativeTA} demonstrated that if these classes were relativized to an oracle drawn at random (i.e., one in which every string is either in or out with probability $\frac{1}{2}$) then $\bm{P} \neq \bm{NP}$ with probability $1$. 
They saw this as potentially strong evidence that $\bm{P} \neq \bm{NP}$ were not equal in our reality. Upon showing similar results for other pairs of complexity classes, they closed the paper by framing the random oracle hypothesis: that if two classes demonstrated a certain relationship with probability $1$ in the sense of drawing an oracle at random, then that relationship ought to hold in reality. \par 
As early as 1983 \cite{kurtz1982}, this hypothesis had been shown to be false in a variety of ways, 
most strikingly by Chang in 
1994 who demonstrated that $\bm{IP} \neq \bm{PSPACE}$ with probability $1$\cite{CHANG199424}, 
despite Shamir's result just two years earlier proving that $\bm{IP} = \bm{PSPACE}$ 
unrelativized \cite{IPequalsPSPACE}. In their original paper, Bennet and Gill anticipated that 
their hypothesis was likely false, and that the condition 
might have to be strengthened. They wrote:   
	\begin{quote}
		 We believe that this hypothesis, \textit{or a similar but stronger one}, 
captures a basic intuition of the pseudorandomness of nature from which many apparently 
true complexity results follow.[Our emphasis]
	\end{quote}

We propose an alternative hypothesis, involving a different notion of ``largeness'' 
which seems to work in the cases we consider in the sense that if two classes are 
separated relative to a large set of oracles, then the unrelativized classes are distinct. 
Furthermore, the set of oracles relative to which the complexity classes we consider are 
separated is not small. We note that other notions or largeness have also been considered 
by various authors. For example, it was shown (see for example \cite{KM}) that 
these separating sets of oracles were large in the category sense, that is, they are comeager in the 
standard topology on the Cantor space $2^\omega$. The notion we consider is the notion of 
comeagerness with respect to the Ellentuck topology (defined in \S\ref{section:Ramsey}) on $2^\omega$, identified with 
$\sP(\omega)$ or equivalently $[\omega]^\omega$, the set of increasing functions from $\omega$ to $\omega$
(we ignore finite sets here, as they correspond to trivial oracles). 
The Ellentuck topology is important and arises naturally in 
logic and set theory, and been extensively studied. It is closely related to a forcing notion 
called Mathias forcing, and this provides a sometimes convenient alternate point of view. 
The Ellentuck topology, or Mathias forcing, naturally arises in studying infinitary partition properties 
on $\omega$. The classical Ramsey theorem says that $\omega$ satisfies the finite exponent 
partition property, that is $\omega \to (\omega)^{k}_\ell$ for any $k,\ell \in \omega$, in 
the Erd\"os-Rado notation (see \S\ref{section:Ramsey} for a more complete statement).  
Assuming the axiom of choice, $\ac$, $\omega$ cannot satisfy an infinite exponent partition 
relation, however it can hold for certain classes of definable sets. The Galvin-Prikry theorem 
asserts that for Borel subsets $A\subseteq [\omega]^\omega$, the infinite exponent partition 
property holds (again, see \S\ref{section:Ramsey} for complete statements). This was extended 
by Silver to all analytic sets ($\bS^1_1$) sets. Mathias introduced a forcing notion
which could give these results, and Ellentuck reformulated this a topological notion, the 
Ellentuck topology. Sets $A\subseteq [\omega]^\omega$ for which a strong form of the 
partition property holds (the completely Ramsey sets, see \S\ref{section:Ramsey})
correspond to sets which are topologically regular, that is have the Baire property, with respect to
this topology. The notions of measure, category, and largeness with respect to category in the 
Ellentuck topology, which we call {\em Ramsey large}, are thus natural notions to use to measure 
the size of sets of oracles. One of the goals of this paper is to investigate the sets of oracles providing
separations between the various computational complexity classes with respect to the notion 
of Ramsey largeness.

In Theorem~\ref{Ohomo} and Corollary~\ref{cor:Ohomo} we show that the set
$\bm{O}$ of oracles separating $\bm{P}$ and $\bm{NP}$, and the set $\bm{O'}$
of oracles relative to which $\np \neq \cnp$ are both Ramsey positive 
(that is, they are not Ramsey small). This result agrees with the measure and category 
versions, however unlike those cases we are not able to show these sets are actually Ramsey large. 
Moreover, in Theorem~\ref{orl} we show that if $\bm{O}$ is Ramsey large then 
$\bm{P}\neq \np$. Thus, the version of the random oracle hypothesis using 
Ramsey largeness instead of measure holds for the class $\bm{O}$. 
In Theorem~\ref{Qhomo} we consider quantum computational complexity classes, and show 
that the oracles relative to which $\np \nsubseteq \bm{BQP}$ is also 
Ramsey positive. In Theorem~\ref{qrl} we obtain the corresponding version of the 
random oracle hypothesis, that is, we show that if this set is Ramsey large,
then  $\np \nsubseteq \bm{BQP}$. In Theorem~\ref{Shomo} we show that the situation for 
the classes $\bm{PSPACE}$ and $\bm{PH}$ is somewhat different, Namely, we show that
if the set of oracles separating these two classes in Ramsey positive, then 
$\bm{PSPACE}\neq \bm{PH}$. It thus remains an open question whether this set
is Ramsey positive, however this result shows that the Ramsey version of the random oracle hypothesis 
holds for these classes as well.  
We conclude the section by discussing what this survey reveals about the ``Ramsey oracle hypothesis''
being true more generally.


A second goal of the paper is consider the descriptive complexity of the various sets of oracles
separating these computational complexity classes. Here the notions of descriptive set theory provide 
a standard way of calibrating the complexity of a set of reals (oracles). 
Recall that in any topological space $X$ the collection of Borel sets is defined as the 
smallest $\sigma$-algebra containing the open sets. This is stratified into a hierarchy, the Borel 
hierarchy, as follows. We let $\bS^0_1$ denote the open sets, $\bP^0_1$ the closed sets, and 
$\bD^0_1=\bS^0_1\cap \bP^0_1$ the clopen sets. In general, for $\alpha<\omega_1$
a countable ordinal we let $\bS^0_\alpha$ be the sets $A$ which can be written as a countable
union $A=\bigcup_{n \in \omega} A_n$ where each $A_n \in \bP^0_{\alpha_n}$ for some $\alpha_n<\alpha$.
Similarly, $A\in \bP^0_\alpha$ if $A=\bigcap_{n \in \omega} A_n$ where each $A_n\in \bS^0_{\alpha_n}$
for some $\alpha_n <\alpha$. Also, we let $\bD^0_\alpha=\bS^0_\alpha \cap \bP^0_\alpha$. 
It is a standard fact that in any uncountable Polish space (completely metrizable, separable space)
the levels of this hierarchy do not collapse, that is $\bS^0_\alpha \neq \bD^0_\alpha$ for all $\alpha$. 
These levels of the Borel hierarchy thus give a way to measure the exact complexity of a set. 
We say a set $A$ is $\bS^0_\alpha$-hard if $A \notin \bP^0_\alpha$, which says that $A$ 
is at least a complicated as sets at the $\bS^0_\alpha$ level. We likewise say $A$ 
is $\bP^0_\alpha$-hard if $A\notin \bS^0_\alpha$. We say the set $A$ is $\bS^0_\alpha$-complete 
if $A \in \bS^0_\alpha$ and is $\bS^0_\alpha$, and likewise define $A$ being $\bP^0_\alpha$-complete. 
Being $\bS^0_\alpha$-complete says that the set is at the exact complexity of $\bS^0_\alpha$, that is,
it is $\bS^0_\alpha$ but not any simpler (not $\bD^0_\alpha$). Similarly for 
being $\bP^0_\alpha$-complete. Computing  upper-bounds for the complexity of a set is typically
much easier than proving the lower-bounds.

In Theorems~\ref{ocomu} and \ref{Ohardness} we show that the set $\bm{O}$ of oracles separating $\bm{P}$
and $\np$ is $\bP^0_2$-complete. We similarly show that $\bm{O'}$, the set of oracles 
separating $\np$ and $\cnp$ is $\bP^0_2$-complete. In Theorems~\ref{Qub} and \ref{Qhard}
we show that the set $\bm{Q}$ of oracles relative to which $\np \nsubseteq \bqp$
is also $\bP^0_2$-complete. Finally, in Theorems~\ref{Sub} and \ref{Shard} we show that the set $\bm{S}$
of oracles separating $\ps$ and $\ph$ is $\bP^0_2$-complete. 

For the rest of this section we fix some notation we will use for the rest of the paper. 
We let $\omega=\N=\{0,1,2,\dots\}$ denote the natural numbers. The Baire space $\omega^\omega$ is the 
space of functions from $\omega$ to $\omega$ with the product of the discrete topologies on $\omega$. 
$[\omega]^\omega$ is the set of increasing function from $\omega$ to $\omega$. $2^\omega$
denotes the Cantor space of functions from $\omega$ to $2=\{ 0,1\}$ with the product of the discrete 
topologies. Both $\ww$ and $2^\omega$ are Polish spaces, with $2^\omega$ being compact. 
We identify $2^\omega$ with $\sP(\omega)$, the power set of $\omega$ by considering 
characteristic functions. We also identify $[\omega]^\omega$ with the set of all infinite 
subsets of $\omega$ via their enumerating functions.


We fix for the rest of the paper an alphabet which we take to be the standard two-symbol 
alphabet $2=\{0,1\}$ unless otherwise specified. By a {\em string} we mean a finite sequence 
from the alphabet, that is, an element of $2^{<\omega}$. By a {\em language} we mean a subset
of strings. We fix a reasonable effective enumeration of the strings, which we 
write as $s_0,s_1,\dots$. This way, the set of strings is identified with $\omega$, and 
a language is identified with an element of $2^\omega$. 
An {\em oracle} refers to a particular language, that is, a specific subset of strings. 
Thus, oracles can also be identified with elements of $2^\omega$. 
As in the previous paragraph, we may also 
identify infinite languages or oracles with elements of $[\omega]^\omega$. 
In this way, natural notions such as measure 
or category on $2^\omega$ apply to sets of oracles. Similarly the notion of Ramsey large 
on $[\omega]^\omega$ carries over to sets of infinite oracles. 
The exclusion of finite language or oracles in the $[\omega]^\omega$ case 
has no implications of significance regarding our results. 
Monotonicity is a fundamental property of any reasonable notion of measure. 
If a class of languages turns out to be ``large,'' in any sense, then it should and will 
remain large when we restrict to infinite oracles. 
Thus, throughout the paper we make no distinction between complexity classes and their 
slightly thinned out counterparts in Ellentuck space.

We introduce the classes of oracles we will be primarily concerned with in this paper.
Recall that for a complexity class such as $\bm{P}$, $\np$, etc., the relativized 
class $\bm{P}^A$, $\np^A$ are are sets of languages that are in the class defined using $A$ as an 
oracle. We will use the following notation for the sets of oracles providing separations 
of various classes.

\begin{align}
	\bm{O} = \{A: \bm{P}^A \neq \bm{NP}^A\}
\end{align}
\begin{align}
	\bm{O'} = \{A: \bm{NP}^A \neq \cnp^{A}\}
\end{align}
\begin{align}
	\bm{S} = \{A: \bm{PH}^A \neq \bm{PSPACE}^A\}
\end{align}
\begin{align}
	\bm{Q} = \{A: \bm{NP}^A \nsubseteq \bm{BQP}^A\}
\end{align}

Before continuing to these results, we fix some common notations and conventions for later reference. 
If $M$ is a Turing machine presumed to halt on all inputs, we will write $L(M)$ to denote the language 
decided by that machine. For a string $x$, let $|x|$ denote it's length. If $\sigma$ is a symbol over some alphabet, 
we will write $\sigma^n$ to denote the string that is simply $\sigma$ repeated $n$ times. 
We will for the rest of the paper 
fix a computable enumeration over the set of all deterministic Turing machines $\{M_i\}$, and 
another for all nondeterministic Turing machines $\{N_i\}$. For both enumerations, 
we assume that every machine appears infinitely often (perhaps with padding). 
We will commonly write TM as an abbreviation for Turing machine, and NTM as an abbreviation 
for nondeterministic Turing machine. 


\section{Descriptive Complexity}\label{section:descriptiveComplexity}
It turns out that most sets of oracles separating two different complexity classes are easily 
shown to be $\Pi^0_2$ sets--the following argument, applied first to $\bm{O}$, 
could easily be applied to many others, as we will see. For the hardness 
results in this section  we will need to make use of a language powerful 
enough to force equality between the classes we are considering, when thought of as an oracle. 
Thus for the record we state the following lemma, whose proof can be found in any complexity theory textbook.

\begin{lemma}\label{squeeze}
Let $E$ be any $\bm{EXP}$-complete language. Then 
\[ \bm{P}^E = \bm{NP}^E = \bm{BQP}^E = \bm{PH}^E = \bm{PSPACE}^E = \bm{EXP} \]
\end{lemma}

In the following subsections we determine the exact complexity of the sets of oracles separating the 
various pairs of classes we consider.

\subsection{P and NP}

First, we have the following upper-bound on the complexity of $\bm{O}$. A straightforward 
computation appears at first glance to give $\bm{O}\in \Sigma^0_3$. The following 
result shows it is actually $\Pi^0_2$. 

\begin{theorem} \label{ocomu}
$\bm{O} \in \Pi^0_2$. 
\end{theorem}

\begin{proof}
Define the following functions 
\begin{align}
\mathcal{N}_a(i,n,k) = \begin{cases}
							1 & \textrm{if $N_i(x_n)$  halts in acceptance in time less than $|x_n|^k$} \\
							0 & \textrm{otherwise} 
						 \end{cases}
	\end{align}
	\begin{align}
	\mathcal{N}_t(i,n,k) = \begin{cases}
							1 & \textrm{if $N_i(x_n)$  halts in time less than $|x_n|^k$} \\
							0 & \textrm{otherwise} 
						 \end{cases}
	\end{align}
	\begin{align}
	\mathcal{M}_a(j,n,k) = \begin{cases}
								1 & \textrm{if $M_j(x_n)$ halts in acceptance in time less than $|x_n|^k$} \\
								0 & \textrm{otherwise} 
							 \end{cases}
	\end{align}
	\begin{align}
	\mathcal{M}_t(j,n,k) = \begin{cases}
								1 & \textrm{if $M_j(x_n)$ halts in less than $|x_n|^k$ steps} \\
								0 & \textrm{otherwise} 
							 \end{cases}
\end{align}

$\mathcal{M}_a$ and $\mathcal{M}_t$ are both clearly recursive (computable) 
To be more formal about $\mathcal{N}_t$ and $\mathcal{N}_a$: what we mean by ``$N_i(n)$ halts in acceptance'' 
is that there exists an accepting path in the configuration graph of $N_i(n)$. 
What we mean by ``$N_j(n)$ halts in time less than $|n|^k$ steps'' is that all paths of 
uniquely occurring vertices in the configuration graph of $N_i(n)$ (i.e., all of those paths without loops) 
have length less than $|n|^k$. Since both functions only require checking a finite number of paths, 
both of these functions $\mathcal{N}_a$ and $\mathcal{N}_t$ are recursive. 
For all of these functions, denote $\mathcal{M}_t^A$ and so forth to be the same functions, 
but with machines relativized to $A$.

The naive way to describe an oracle being in $\bm{O}$ is the following:
\begin{equation} 	
 	\begin{split}
		 A \in \bm{O} \iff \exists i,k_1 \forall j,n_1 [ \mathcal{N}^A_t(i,n_1,k_1) & \wedge ( ( \exists k_2 \forall n_2 \mathcal{M}^A_t(j,n_2,k_2) ) \\ & \Rightarrow \exists n_3 (\mathcal{N}^A_a(i,n_3,k_1) \neq \mathcal{M}^A_a(j,n_3,k_2)) ) ]
	\end{split}	 
\end{equation}
This sentence makes the claim that there exists some relativized nondeterministic Turing machine $N_i^A$ which halts in time polynomially bounded (with degree $k_1$) by its input lengths for all inputs 
($n_1$ being the input), meaning that it decides some language $L_1 \in \textbf{NP}^A$, 
and which has the property that if $M_j^A$ is any polynomial time bounded Turing machine (of degree $k_2$) 
(meaning it decides some language $L_2 \in \textbf{P}^A$), then there must exist an input ($n_3$) 
on which $N_i^A$ and $M_j^A$ disagree (meaning that $L_1 \neq L_2$).  Putting this in 
prenex normal form pulls out another existential quantifier, yielding a $\Sigma^0_3$ expression for $\bm{O}$.

We can do better than this, however. The trick is to talk about equality of classes 
in terms of the existence of a ``universal'' language which, were it in the weaker of the two classes, 
would render that class just as powerful. We define such a language in terms of a ``universal'' nondeterministic 
Turing machine $U$. The machine $U$ takes inputs of the form $x,p^n,m$, 
where $x$ is a string, $p$ is a special symbol which should never occur in $x$ or $m$, 
and $m$ codes an integer index for a nondeterministic Turing machine. On an input of this form, 
the machine $U$ simulates the $m^{th}$ nondeterministic machine on the input $x$, for $|x|+n$ many steps, 
and accepts if the $m^{th}$ machine accepts in this time, rejecting otherwise. Since there exists 
a universal nondeterministic Turing machine which operates in linear simulation overhead, with 
minor modification we can easily create a machine $U$ which works as described above and halts in polynomial time. 
Thus the language decided by $U$, $L(U)$, is clearly in $\bm{NP}^A$, as is it's relativized 
counterpart $L(U^A)$ in $\bm{NP}^A$, for any fixed oracle $A$.

Suppose that for some oracle $A$, $L(U^A)$ were in $\bm{P}^A$. Let $M^A_U$ be the poly-time machine 
which decides it. It is simple enough to check that $\bm{P}^A = \bm{NP}^A$: Let $N^A$ be an arbitrary 
relativized nondeterministic Turing machine which halts in time $n^k$ for some $k \in  \omega$, 
and let $m$ be the integer index of this machine. Fix an input $x$. To decide if $x \in L(N^A)$, 
we can simply simulate the machine $M^A_U$ on the input $x;p^{|x|^k-|x|};m$, and accept or reject accordingly. 
Clearly this process can be accomplished in polynomial time, and thus we have shown that $L(N^A) \in \bm{P}^A$. 
Thus we have that $L(U^A) \in \bm{P}^A$ iff $\bm{P}^A = \bm{NP}^A$. 
Let $u \in \omega$ be the index of the nondeterministic machine $U$, and let $l \in \omega$ 
be the polynomial bound on that machine. Stated in terms of the functions we have defined earlier, we have

\begin{align} \label{Pi2}
A \in \bm{O}^c \iff \exists i,k \forall n [(\mathcal{M}^A_t(i,n,k)=1) 
\wedge (\mathcal{N}^A_a(u,n,l)=1 \iff \mathcal{M}^A_a(i,n,k)=1)] 
\end{align}

Thus $\bm{O}^c \in \Sigma^0_2$, and so $\bm{O} \in \Pi^0_2$.
\end{proof}

Since the claim that ``there exists a machine which decides this language'' is generally always going 
to be a $\Sigma^0_2$ statement, it's clear that such an argument should work 
for all complexity classes in which it's known that a universal machine exists. 
In particular, an identical argument confirms 
that $\bm{O'} \in \Pi^0_2$. The case of $\bm{Q}$ will also show itself to be a nearly 
identical argument, we give the details in section \ref{descriptiveQ}. The case of $\bm{S}$ 
requires a bit more discussion, but nonetheless we will show it to be $\Pi^0_2$ 
as well in section \ref{descriptiveS}.  Since the lightface classes are contained 
in their boldfaced counterparts, we have the following weaker corollary which will be useful to our Ramsey results:

\begin{corollary}
$\bm{O}$ and $\bm{O'}$ are Borel.
\end{corollary}

We next show that $\bm{O}$ and $\bm{O'}$ are $\bP^0_2$-hard, and thus are $\bP^0_2$-complete.

\begin{theorem}\label{Ohardness}
$\bm{O}$ is $\bP^0_2$ hard. 
\end{theorem}

\begin{proof}

We use the set 	
\begin{equation} \label{infinitelyOne}
H = \{b \in 2^\omega \colon b(n)=1 \textrm{ for infinitely many } n\},
\end{equation}
which is well-known 
(and easy to see) to be $\bP^0_2$-complete (see \cite{KechrisBook}).
To show that $\bm{O}$ is $\bP^0_2$ hard it suffices to define a {\em reduction} 
$f$ of $H$ to $\bm{O}$. That is, $f\colon 2^\omega \to 2^\omega$ is a continuous function 
such that $b \in H$ iff $f(b) \in \bm{O}$. 
Note that despite having the same domain and range, we will be viewing elements of the domain as infinite 
binary sequences, and elements of the codomain as sets of strings over a fixed alphabet.

Thus given a binary sequence $b$, we want to construct an oracle $A$ such that $\bm{P}^A \neq \bm{NP}^A$ 
iff $b$ is $1$ infinitely often. Let $E$ be any $\bm{EXP}$-complete language. The strategy 
is to perform a ``delayed'' version of Solovay's original diagonalization out of 
$\bm{P}$ \cite{solovayOriginal}, in which the $n^{th}$ stage of that construction only happens 
when we encounter the $n^{th}$ $1$ of the sequence. Whenever the $n^{th}$ bit is a $0$, we'll 
add all strings from $E$ of length $i$ to the oracle, so that finitely many $1$'s corresponds 
to an oracle which is ``as powerful'' as $E$ up to a finite number of answers which can be hard coded. 
Therefore by lemma \ref{squeeze}, we will be left with an oracle $A$ such that $\bm{P}^A = \bm{NP}^A$. 
With an infinite number of $1$'s, we will likely still be left with a significant amount of strings of $E$, 
but we will have nonetheless fully diagonalized out of $\bm{P}^A$.

We now describe the construction explicitly. Fix a bit string $b$. For any $n$, let $d(n)$ 
denote the number of $1$'s that have been encountered in $b$ by the $n^{th}$ digit (inclusive).
We construct our oracle $A$ in stages, corresponding to the bits of $b$.
Let $A_{n-1}$ be what we have constructed of the oracle after stage $n-1$. 
We commit that we will never add strings of length $i < n$ at stage $n$.

First, if $b(n)=0$, then we simply add all strings from $E$ of length $n$, 
and move to the next stage, i.e., $A_n$ will equal $A_{n-1}$ unioned with all strings of length $n$ from $E$. 
Second, if $b(n)=1$, then we run $d(n)^{th}$ Turing machine relativized to 
$A_{n-1}$ for $n^{\log(n)}$ steps on the input $1^n$. Suppose during this computation that the 
machine queries a string $x$. If $|x| < n$, then the decision of whether $x \in A_{n-1}$ will 
have already been settled. Suppose $|x| \geq n$. If this is the case and $b(|x|)=1$, 
then we make note of this by adding $x$ to an ongoing list of exceptions $X$ (these 
are strings which we are not allowed to add at a later stage so as to not alter the computation).
In running this computation, we proceed as if $x \notin A$, which will be the case if we follow our
commitments regarding the set $X$. 
The goal of the stages corresponding to $b(n)=1$ is to make it so that 
no polynomial time deterministic Turing machine can decide the language

\begin{equation}\label{solovayLang}
L_A = \{1^n: A \textrm{ contains a string $x$ of length $n$} \} 
\end{equation}

\noindent
We seek to do this by adding a single string $x$ of length $n$ to $A_{n-1}$ in the case of the 
machine rejecting, and in the case of the machine not rejecting we add nothing, so as to ensure 
that the machine makes an error on $1^n$. If the computation 
$M^{A_{n-1}}_{d(n)}(1^n)$ does not halt in rejection in at most  $n^{\log(n)}$ steps then 
we add no new strings at stage $n$, that is, we set $A_n=A_{n-1}$. Suppose 
$M^{A_{n-1}}_{d(n)}(1^n)$ does halt in rejection within this number of steps. 
In the worst case scenario, every machine $M^{A_0}_1,\ldots,M^{A_{n-1}}_{d(n)}$ will have queried a 
different string of length $n$ at every step of it's computation. In this case the set of strings $X$ 
of length $n$ which we've committed to not adding to our oracle is $\sum_{j=1}^n j^{\log(j)}$, 
which is less than $2^n$ for all $n \in \omega$. Thus, there will be a string $x$ of length $n$
which is not in the set $X$. We let $A_n=A_{n-1}\cup \{ x\}$, where $x \notin X$ is the 
lexicographically least  string not in $X$ such that $|x| = n$. 
This completes the construction of the $A_n$. We let $A=\bigcup_n A_n$. 
We let $f(b)$ be the oracle $A$ just constructed. We have thus defined the map 
$f \colon 2^\omega \to 2^\omega$.

To see this works, first suppose $b(n)=1$ for only finitely many $n$. 
Then $A = (E \sm P) \cup Q$ for some finite sets of strings $P$ and $Q$. 
Clearly this is still $\bm{EXP}$ complete, and so $\bm{P}^A = \bm{NP}^A$. 
On the other hand, suppose that $b(n)=1$ for infinitely many $n$. Suppose by way of contradiction 
that there is a polynomial time relativized deterministic Turing machine $M^A$ which decides $L_A$. 
Let $m$ be an index for it in our enumeration and $k$ be the degree of it's polynomial bound. 
Let $j \in \omega$ be large enough that for all $i \geq j$, $i^k < i^{\log(i)}$. By our assumptions 
there exists an $l \geq j$ such that $b(l)=1$ and $M^A_m = M^A_{d(l)}$. 
Thus at stage $l$ of the construction, we observed the computation of $M^A_m$ on the input $1^l$ 
for a number of steps sufficient for it to halt and make a decision. Suppose it halted in rejection. 
Then at stage $l$ we added a string $x$ of length $l$ to $A=f(b)$. So $1^l \in L_A$, a contradiction since $M_m$ 
is supposedly deciding $L_A$. Suppose it halts in acceptance. Then by construction $A$
has no strings of length $l$ in it.  Thus $1^l \notin L_A$, and our machine $M_m$ is again in error. 
We have shown that no polynomial time machine can decide $L_A$.

On the other hand, $L_A \in \bm{NP}^A$ regardless of $A$, because we can simply have a 
nondeterministic machine guess a string of the proper length and check with the oracle to 
confirm membership, in $O(1)$ steps. Thus we have  established that $\bm{P}^A \neq \bm{NP}^A$ iff $b(n)=1$ 
for infinitely many $n$. That is,  $b \in H$ iff $f(b)\in \bm{O}$. 

It remains to show that this function is continuous. 
To show this it suffices to show that for any $b \in 2^\omega$, there exists a function 
$g\colon \omega \to \omega$ such that, for any string $x$, membership of $x$ in $f(b)$ depends on 
at most $g(|x|)$ bits of $b$. Fix $b$ and suppose $x$ is a string of length $n$.
We can decide if $x \in f(b)$ by running through the first $n$ stages of the construction of $f(b)$. 
If $b(n) = 0$, then $x \in f(b)$ iff $x \in E$, meaning that only the first $n$ bits of $b$ are 
required to determine membership. If $b(n) = 1$, then we would add $x$ to $f(b)$ 
iff the machine $M_{d(n)}^{f(b)}$ halts in rejection on $1^n$, and $x$ is the smallest string 
of length $n$ that has not been queried at any stage $i \leq n$ such that $b(i)=1$. It will likely 
be the case that this computation will need to query strings of length greater than $n$. 
However, there are finitely many such strings which are queried. Let $g(n)$
be the maximum length of these queried strings for this and the previous computations.
Then membership of $x$ in $A$ depends 
only on the first $g(n)$ digits of $b$.
\end{proof}

\begin{corollary}
$\bm{O}$ is $\bP^0_2$-complete.
\end{corollary}

\subsection{NP and BQP}\label{descriptiveQ}
It is clear that the argument of Theorem~\ref{Ohardness} 
can be potentially recycled for any pair of classes for which there 
is a diagonal process of constructing oracles separating them. Another example of this is the 
separation of $\bm{NP}$ from $\bm{BQP}$, and so we turn our attention to quantum Turing machines. 
The following results assume the definitions and conventions of Bernstein and Vazirani's oracle 
quantum Turing machine model from their landmark paper hammering out the technical details of 
these machines \cite{quantumComplexityTheory}, and uses their notation. For a comprehensive 
overview of quantum Turing machines, we refer the reader to that paper. In the interest of completeness, 
however, we supply a brief overview. Quantum Turing machines are a generalization 
of the deterministic Turing machines (with two-way infinite tape) in which the transition 
function takes symbol-state pairs not to particular actions, but rather to vectors in which 
an efficiently computable complex number is assigned to every possible action which a deterministic machine 
could take in that context (the set of possible complex numbers which the transition function 
assigns can in fact be assumed finite). Configurations are seen as orthonormal basis vectors in 
a countable dimensional vector space populated by finite linear combinations of these. 
The machine with it's modified transition function implicitly defines a transformation from 
these basis vectors to finite superpositions of them, called the \emph{time evolution operator}, 
and this can be extended to a full linear transformation on the space called the 
\emph{linear time evolution operator}.  The operator defined by the machine takes a fixed 
initial configuration to a finite superposition of possible configurations (one for each action, 
weighted according to the complex number given by the transition function), and 
further steps are seen as repeated application of the same operator. 
The squared magnitude of the complex number weighting a configuration 
of the superposition is seen as giving the probability that one will observe that configuration 
upon ``measurement'' of the tape. The actual set of quantum Turing machines requires restrictions 
to be consistent with actual physics. In particular, the QTMs which should be viewed as legitimate 
are those for which the linear time evolution operator is unitary. Bernstein and Vazirani derive 
in their paper a finite set of local conditions directly on the transition function which are 
necessary and sufficient for this to be the case. This means that we can recursively enumerate 
over the set of all QTMs just as easily we can over DTMs, NTMs, or PH expressions.

Halting configurations (QTMs are defined to have a single halting state) deserve special mention. 
The convention is to say that a QTM halts in $T$ steps if it after $T$ applications of the 
linear time evolution operator, the machine is left in a superposition such that all non-zero 
configurations are halting configurations, and furthermore no halting configuration appears 
in any superposition up to that many applications. Thus the question of how to interpret superpositions 
in which some configurations are halting, and others are not, is avoided. Input conventions, i.e., 
specifying the initial configuration for a desired input string $x$, remain identical to 
deterministic Turing machines. A QTM is said to operate in time $f(n)$ if it halts on all 
strings of length $n$ in time less than or equal to $f(n)$. A QTM is said to accept the string 
$x$ with probability $p$ in $T$ steps if it halts in $T$ steps from the appropriate initial 
configuration, and additionally if one measures the starting cell of the tape at this point, 
they will see a $1$ with probability $p$, and a $0$ with probability $1-p$ (we will assume that 
the alphabet for all QTMs include at the very least $0$, $1$, and a distinct blank symbol). 
Rejection is defined identically. An entire language $L$ is accepted by a QTM with probability $p$ 
if the machine accepts every string $x \in L$ with probability at least $p$, and also rejects 
every string $x \notin L$ with probability at least $p$. Finally the class $\bm{BQP}$ is defined 
to be the class of languages $L$ for which there exists a $QTM$ $M$, and an integer $k \in \omega$ 
such that $M$ operates in time $O(n^k)$ and accepts $L$ with probability $\frac{2}{3}$. 
Multitrack QTMs are also defined in the obvious way.

Turning to oracles, an oracle QTM has a special query track, and two special query states: 
a prequery state and postquery state. A query configuration is one in which the machine is in a 
prequery state, and the query tape has a single uninterrupted nonempty string (i.e., a single 
sequence of non-blank characters) of the form $x\conc i$, where $i$ is a $0$ or a $1$. 
Thinking of the oracle as a function outputting a $0$ or a $1$, application of the linear time 
evolution operator on a basis configuration of this form will always yield a configuration 
in which everything is unaltered except for the state, which becomes the postquery state, 
and the query tape, which will display $x$ immediately followed by $i \bigoplus f(x)$, 
where $\bigoplus$ denotes addition modulo $2$. This is important for keeping the transformation unitary,
although for our purposes we will always assume that $i$ is $0$, so that the direct answer 
to the query is simply printed next the string we were querying. It should be noted that queries 
are deterministic--a basis state corresponding to a query configuration yields a single 
postquery basis configuration with amplitude $1$. For an oracle $A$ and a QTM $M$, 
$M^A$ denotes the machine $M$ relativized to $A$ and $\bm{BQP}^A$ denotes the class of 
languages accepted with probability $\frac{2}{3}$ in polynomial time by a QTM relativized to $A$.

In accordance with the probabilistic convention for accepting and rejecting, simulation of one 
QTM by another, and simulation of arbitrary unitary operations by QTMs, are defined according to 
approximation rather than exactness.
The following lemma from \cite{strengthsWeaknesses} is valuable to both our results and this idea.

\begin{lemma}[Theorem~3.1 of \cite{strengthsWeaknesses}] \label{distDist}
If two unit length superpositions $\ket{\phi}$ and $\ket{\psi}$ are within Euclidean 
distance $\epsilon$ of each other (i.e. $||\phi-\psi|| \leq \epsilon$), then the maximum probability 
distance between the probability distributions resulting from measurement is at most $4\epsilon$. 
That is to say, given any observable event $E$, the probability that $E$ occurs upon measuring 
$\ket{\phi}$ will be no more than $4\epsilon$ from the probability that $E$ 
occurs upon measuring $\ket{\psi}$.
\end{lemma}

By this lemma, two QTMs which are expected to have 
superpositions which are close in Euclidean distance, can be expected to be close in the 
distribution of measured outcomes. In their paper, Bernstein and Vazirani demonstrate that 
an arbitrary unitary transformations on $n$-qubit space can be approximated 
with exponential accuracy in polynomial time 
(weighted appropriately by the dimension $n$).

They also demonstrate the existence of a 
universal QTM, which can simulate arbitrary QTMs with any desired accuracy with slowdown 
polynomial in the runtime of the original machine and also polynomial in the inverse of the 
desired accuracy. Thus we can speak of simulating one QTM by another, or simulating an 
arbitrary unitary transformation. From these arguments 
it is also follows that there is a universal oracle Turing machine $M^A$ 
that can efficiently simulate all other QTMs relativized to $A$.

The ability to recursively enumerate over the collection of all QTMs, along with the existence of 
universal relativized QTMs for any oracle, is enough to import our argument from Equation~(\ref{Pi2}) 
into our current context. Redefine the context of $\mathcal{M}^A_a(j,n,k)$ and 
$\mathcal{M}^A_t(j,n,k)$ so that $M_j$ refers to the $j^{th}$ oracle QTM as opposed to the 
$j^{th}$ oracle DTM, then by these two results we can be sure that these new functions are 
still recursive (relative to a fixed oracle). Have $U$, and the integer $u$ be 
defined as in Equation~(\ref{Pi2}). It is clear that $\bm{NP}^A \subseteq \bm{BQP}^A$ 
iff $L(U^A) \in \bm{BQP}^A$.  We therefore have

\begin{align} 
A \in \bm{Q}^c \iff \exists i,k \, \forall n [(\mathcal{M}^A_t(i,n,k)=1) \wedge 
(\mathcal{N}^A_a(u,n,l)=1 \iff \mathcal{M}^A_a(i,n,k)=1)] 
\end{align}
Thus we have 

\begin{theorem}\label{Qub}
$\bm{Q}$ is $\Pi^0_2$ (and therefore also Borel). 
\end{theorem}
 
Superpositions of configurations will be denoted using `bra-ket' notation. 
For a superposition $\ket{\phi}$ and a string $y$, the \emph{amplitude measure} $q_y(\phi)$ 
is defined to be the sum of the squared magnitudes of all amplitudes corresponding to query 
configurations intending to ask the oracle about $y$. For our hardness results, as well as our 
results pertaining to Ramsey measure, we borrow strategies from Bennett, Bernstein and Vazirani's 
follow-up paper \cite{Bennett1981RelativeTA}. In particular we makes use of one of the lemmas found there. 

\begin{lemma}[Theorem~3.3 of \cite{strengthsWeaknesses}] \label{oracleAlter}
Fix a quantum Turing machine $M$, an oracle $A$, and an input $x$. For any natural number $i$, 
let $\ket{\phi_i}$ denote the current superposition for $M^A(x)$ at step $i$. Let $\epsilon > 0$. 
Let $F \subseteq [0,T-1] \times \{0,1\}^*$ be a set of time-string pairs such that 
$\sum_{(i,y) \in F} q_y(\phi_i) \leq \frac{\epsilon^2}{2T}$. 
Suppose that for each $(i,y)\in F$ the answer to the query $y\in A$ at time $i$ 
is changed to an arbitrary value. Let $\ket{\phi_i'}$ be the superposition at time $i$
with these modified responses.  
Then $|\ket{\phi_T}-\ket{\phi_T'}| \leq \epsilon$. 
\end{lemma}

See \cite{strengthsWeaknesses} for proofs of these results. This result, taken along with 
Lemma~\ref{distDist}, imposes a bound on the extent to which a single string added to an oracle can 
change the overall superposition after a certain number of steps. We will use this to guarantee 
that by adding certain strings to our oracle, it won't ``perturb'' the superposition enough to 
flip any answers from acceptance to rejection or vice versa. 

\begin{theorem} \label{Qhard}
$\bm{Q}$ is $\bP^0_2$ hard.
\end{theorem}

\begin{proof}
	
Fix an enumeration of all quantum Turing machines. Again fix an $\bm{EXP}$-complete problem $E$. 
Just as before, given a bit sequence $b\in 2^\omega$, we will conduct a step of diagonalization for each $1$ 
in the sequence, and simply add appropriate strings from $E$ otherwise. However, different from before 
is the number of strings we will potentially have to commit to not adding in order to maintain 
consistency of the construction. Normal deterministic Turing machines can only query polynomially 
many strings in a polynomial amount of runtime. However, when considering the superposition obtained 
from a polynomial number of steps of a quantum machine, we will typically be left with a superposition 
of exponentially many configurations, all of which could be intending to querying something different. 
We will encounter a similar problem for nondeterministic simulations. We can get around this difficulty 
by simulating our machines on very large inputs in comparison to the inputs simulated in previous 
stages of the construction. This will ensure that any string which we might want to add at a current 
stage will simply be too big to have appeared as a query in any simulation prior, leaving us free to 
add whatever we want. As before, we will be simulating for quasipolynomial runtimes. 
We fix an increasing function $g\colon \omega \to \omega$ such that for all $n$, 
$g(n+1) > g(n)^{\log(n)}$. This can be realized explicitly by $g(n) = 2^{2^{2^n}}$ 
(but $2^{2^n}$ is not enough!). This function $g$ is fixed for the rest of the 
argument.

As before, we define a reduction $f$ of $H$ to $\bm{Q}$. Fix $b \in 2^\omega$, and we 
construct an oracle $A$ which will be the value for $f(b)$. The oracle $A$ is again constructed in stages,
and $A_n$ will denote the finite oracle constructed after stage $n$. 
We will have that the $A_n$ are increasing with $n$.  
At stage $n$ we will only add strings of lengths $k$ with $g(n) \leq k <g(n+1)$. 

If $b(n)=0$, then we add to our oracle all strings 
from $E$ of length $k$ such that $g(n) \leq k < g(n+1)$. Note that if $M$ is any QTM, then the 
superposition of $M(1^{g(n-1)})$ after at most $g(n-1)^{\log(g(n-1))}$ computation steps can only contain query 
configurations on strings bounded below this length, and this length by definition of $g$ is 
less than the length of any strings we are adding at stage $n$. From this it will be clear that adding these 
strings from $E$ will have no bearing on any prior simulation.

If $b(n)=1$, then we run the $d(n)^{th}$ QTM on the input $1^{g(n)}$ for 
$T = g(n)^{\log(g(n))}$ many steps relativized to $A_{n-1}$,
where $d(n)$ as before is the number of $m \leq n$ such that $b(m)=1$. 
In constructing $A_n$ we will add at most one string to $A_{n-1}$, and this string (if it exists)
will be of length $g(n)$. Let $\ket{\phi}_i$ denote the superposition after $i$ many steps, relativized 
to $A_{n-1}$. We look for the least $i \leq T$ such that $\ket{\phi}_i$ has a $>\frac{1}{2}$
probability of being in a halting state. If no such $i$ exists, then we set $A_n=A_{n-1}$. 
Suppose such a time exists, and let $i$ be the least such. If $\ket{\phi}_i$ 
has a $>\frac{1}{2}$ probability of being in a halting state which is accepting, then we 
also set $A_n=A_{n-1}$. If $\ket{\phi}_i$ 
has a $\leq \frac{1}{2}$ probability of being in a halting state which is accepting,
then we set $A_n=A_{n-1}\cup \{y \}$, where $y$ is the string constructed below 
such that $| \ket{\phi}_i- \ket{\phi'}_i | < \frac{1}{28}$, where $\ket{\phi'}_i$
denotes the state after running the $d(n)^{th}$ QTM on the input $1^{g(n)}$ for 
$i$ many steps relativized to $A_{n-1}\cup \{ y\}$. Note that the  string $y$ 
will be too large to have appeared in any query 
at a previous stage of the construction. However, such a string could easily 
disturb the current superposition of interest $\ket{\phi}_i$.

We show the existence of a string $y$ of length $g(n)$ 
as mentioned above. 
We will make use of lemma \ref{oracleAlter}. 
With this goal in mind, Let $S$ be the set of strings $y$ of length $g(n)$ such that 

\begin{equation}
\sum_{i=0}^{T-1} q_y(\phi_i) \geq \frac{1}{1568 T}
\end{equation}

Note first that it must be the case that $|S| \leq 1568 T^2$. To see this, 
suppose that $|S| > 1568 T^2$. Then we would have 

\begin{align}
T \geq \sum_{i=0}^{T-1}\sum_{y \in \{0,1\}^*} q_y(\phi_i) &\geq \sum_{i=0}^{T-1} \sum_{y \in S} q_y(\phi_i) \\
&\geq \sum_{y \in S} \frac{1}{1568 T} \\
&> 1568 T^2\frac{1}{1568 T} = T
\end{align}

\noindent
a contradiction. We wish to pull a string $y$ of length $g(n)$ from the complement of this set $S$. 
In the worst case, all strings in $S$ are of this length, but even then the total number of 
strings of length $2^{g(n)}$ far exceeds our upper bound for $S$, and so it follows that there 
exists a $y \notin S$. We take such string $y$ (of length $g(n)$) 
in the definition of $A_n$ above. 
This completes the definition of $A=f(b)$. It remains to show 
that $f$ is a continuous reduction of $H$ to $\bm{Q}$. The fact that $f$ is continuous 
is straightforward as in the proof of Theorem~\ref{Ohardness}. To finish the proof, we show that $b(n)=1$
for infinitely many $n$ iff $A$ is in $\bm{Q}$. Note that regardless of $b$, 
the Solovay language $L_A$ is in $\np^A$.

If $b(n)=1$ for has only finitely many $n$,  then  $A$ equals $E$ up to a finite difference, 
so that $\bm{BQP}^E = \bm{NP}^E$. Thus, $A=f(b) \notin \bm{Q}$.

Suppose next that $b(n)=1$ for infinitely many $n$. Towards a contradiction, 
suppose there is a $\bm{BQP}^A$ machine $M^A$ which decides $L_A$ in time $n^k$ for some $k$.
Fix an $i$ sufficiently large so that $g(i)^k < g(i)^{\log(g(i))}$. 
Since $b(n)=1$ for infinitely many $n$, there is an $l \geq i$ such that $b(l)=1$ 
and $M^A = M_{d(l)}^A$. Thus at stage $l$ of the construction we simulated our machine $M^{A_{l-1}}$ on the 
input $1^{g(l)}$, for a number of steps large enough to ensure that $M^A$ halts.

By assumption, $M^A$ on $1^{g(l)}$ halts at some time $i_0 \leq T$. 
Let $\ket{\phi'}_i$ for $i \leq i_0$ denote the superposition from running 
$M^A$ on $1^{g(l)}$ at time $i$. Note this is the same as running $M^{A_l}$ 
on this input as strings of length $g(l+1)$ are too big to be queried. 
Let $\ket{\phi}_i$ for $i \leq i_0$ denote the state at time $i$ from running 
$M^{A_{l-1}}$ on $1^{g(l)}$. By the conventions on quantum Turing machines, 
$\ket{\phi'}_{i_0}$ is a superposition of configurations, each of which is in a halting state, 
and for all $i<\i_0$, $\ket{\phi'}_{i}$ is a superposition of configurations, none of which is in a 
halting state. From the choice of $y$, Lemma~\ref{oracleAlter} using $\epsilon=\frac{1}{28}$,
and Lemma~\ref{distDist}, that the 
probability that $\ket{\phi}_{i_0}$ is in a halting state is at least $1-\frac{1}{7}>\frac{1}{2}$,
but for $i<i_0$ the probability $\ket{\phi}_{i}$ is in a halting state
is at most $\frac{1}{7}<\frac{1}{2}$. Thus, at step $l$ of the construction, in 
defining $A_l$, we used time $i$. That is, $i=i_0$ is the least time so that 
$\ket{\phi}_{i}$ has a greater than $\frac{1}{2}$ probability of being in a halting 
state.

Suppose first that $M_{d(l)}^A(1^{g(l)})$ halts in an accepting state, that is, with at least 
a $\frac{2}{3}$ probability of acceptance. 
It follows from lemma \ref{oracleAlter} with $\epsilon = \frac{1}{28}$ and 
$F = \{0,1,\ldots,T-1\} \times \{y\}$ that 
that $|\ket{\phi'}_{i_0} - \ket{\phi}_{i_0}| \leq \frac{1}{28}$.
By Lemma~\ref{distDist}, the probability that $M_{d(l)}^{A_{l-1}}(1^{g(l)})$ accepts is within 
$\frac{4}{28} < \frac{1}{6}$ of that of $M_{d(l)}^{A}(1^{g(l)})$ accepting, and thus is 
greater than $\frac{1}{2}$. By the definition of $A_l$ in this case we have that 
$A_l=A_{l-1}$, and so $y \notin A$.  Thus the computation $M_{d(l)}^A(1^{g(l)})$ has made an error.

Suppose next that $M_{d(l)}^A(1^{g(l)})$ halts in a rejecting state, that is, with at least 
a $\frac{2}{3}$ probability of rejection. Then by the same argument of the previous paragraph, 
$M_{d(l)}^{A_{l-1}}(1^{g(l)})$ rejects with probability greater than  $\frac{1}{2}$, and so 
$A_l=A_{l-1}\cup \{ y\}$. So again $M_{d(l)}^A(1^{g(l)})$ has made an error. 

This completes the proof that $f$ is a continuous reduction of 
$H$ to $\bm{Q}$.
\end{proof}

\begin{corollary}
$\bm{Q}$ is $\bP^0_2$ complete. 
\end{corollary}

\subsection{PH and PSPACE}\label{descriptiveS}
Another famous diagonalization argument is the separation of $\bm{PH}$ and $\bm{PSPACE}$. In 1984, Furst, Saxe and Sipster \cite{Furst2005ParityCA} demonstrated that such a separation implied the existence of quasipolynomial size constant-depth circuits deciding the parity function. They themselves were able to demonstrate the non-existence of constant-depth polynomial size parity circuits, but were not able to rule out quasipolynomial ones. This was accomplished one year later by Yao \cite{Yao1985SeparatingTP}, and then followed up by H\r{a}stad who proved the optimal lower bound via his famous switching lemma\cite{hastad}:

\begin{theorem}[H\r{a}stad]\label{lowerBound}
There is an increasing sequence $\{n_k\}_{k \in \omega}$ such that for all $k$, 
there are no depth $k$ parity circuits on $n > n_k$ many inputs which are of size 
less than $2^{\frac{1}{10}^{\frac{k}{k-1}}n^{\frac{1}{k-1}}}$.
\end{theorem}

We will use this result to obtain our result that $\bm{S}$ is $\bP^0_2$-hard. 
We first fix some terminology and prove the upper-bound for the complexity of $\bm{S}$.

Let us call a \emph{PH-expression} an expression of the form
\begin{align} \label{phexp}
	\phi(y) = \exists^px_1 \forall^px_2 \ldots Q^px_l R(x_1,x_2,\ldots,x_l,y)
\end{align}
where $R$ is a polynomial time computable relation with degree $k$ for
some $k$, and the superscript $p$ indicates that we are quantifying
over strings of length less than or equal to $|y|^k$.  An effective
enumeration over all Turing machines can clearly be turned into an
effective enumeration over all PH-expressions, for all $l$. A
language is in $\bm{PH}$ iff it has a PH-expression defining
membership. Likewise, by relativizing $R$ to an oracle $A$, we obtain a definition of the
relativized polynomial hierarchy $\bm{PH}^A$.

\begin{theorem} \label{Sub}
$\bm{S} \in \Pi^0_2$ (and therefore also Borel).
\end{theorem}

\begin{proof}
It is well known that the quantified Boolean satisfiability
problem, which we will denote $QSAT$, is $\bm{PSPACE}$
complete.  However, with
a small modification of the standard reduction witnessing that
$QSAT$ is $\bm{PSPACE}$ hard, we can define a reduction from
$\bm{PSPACE}^A$ to the ``relativized language'' $QSAT^A$,
creating a version of the $QSAT$ problem which is always
complete for $\bm{PSPACE}^A$ for all $A$. With this language
defined, it is clear that $A \notin \bm{S}$ iff $QSAT^A \in
\bm{PH}^A$, a statement that is easily expressible with a
$\Sigma^0_2$ formula. To briefly recall how the normal
reduction works: we have a Turing machine $M$ operating in
space $n^k$ for some $k$, and we wish to create a family of
quantified Boolean formulas
\[ \phi(x) = \exists_{y_1} \forall_{y_2} \ldots Q_{y_{|x|^k}} \psi(y_1,\ldots,y_{n^k},x) \]
such that $x$ is in the language decided by $M$ iff
$\phi(x)$. Since our machine uses space $n^k$, we can be sure
that the configurations of the machine are representable by
strings of length polynomial in $|x|$. First, one define a
base case $\psi_0(A,B)$ representing the claim that
``configuration $A$ yields configuration $B$ in one step or
that $A = B$.'' Then they define $\psi_i(A,B)$ to be the
statement that there exists an intermediary configuration $Z$
such that $\psi_{i-1}(A,Z)$ and $\psi_{i-1}(Z,B)$, i.e. that
there exists a sequence of configurations from $A$ to $B$ of
length less than or equal to $2^i$. This is stated in a clever
way such that only one ``instance'' of the $\psi_{i-1}$ is
needed, so that the number of configurations to be quantified
over is kept linear in $i$. The only modification of this
which is necessary is to add to the base case the extra
complication necessary for checking that a string is in the
oracle. That is to say, within the base case, where we
normally have the Boolean formula representing the claim that
``configuration $A$ yields configuration $B$ in zero or one
steps,'' we have the longer claim:
\begin{itemize}
\item 
``configuration $A$ is not in a query state and
yields configuration $B$ in zero or one steps,
\emph{or}
\item $A$ is in a query state, the string being
queried is in $A$, and $B$ is the configuration
yielded by $A$ given that the query accepts,
\emph{or}
\item $A$ is in a query state, the string being
queried is not in $A$, and $B$ is the configuration
yielded by $A$ given that the query rejects.''
\end{itemize}
Clearly this requires no extra quantification and is only
linearly longer than the original, and just as clearly the
resulting expression will be $\bm{PSPACE}^A$ complete for any
oracle $A$. \par We have then that $A \notin S$ iff $QSAT^A
\in \bm{PH}^A$. Let $\psi^A(x,y)$ be the formula representing
that ``relative to $A$, the $x^{th}$ $\bm{PH}^A$ expression
agrees with $QSAT^A$ on the input $y$,'' which can easily be
seen as recursive using the appropriate modifications of our
$\mathcal{M}$ functions from earlier. We have then that $A
\notin \bm{S}$ iff $\exists x \forall y \psi^A(x,y)$, so that
$\bm{S}^c$ is $\bS^0_2$ complete.
\end{proof}

We now turn to the lower-bound for the complexity of $\bm{S}$.

By fixing a $y$, one can view a PH expression as in (\ref{phexp})  as a constant depth circuit whose inputs
correspond to strings which we may or may not decide to put in the
oracle. For a fixed string length $n$ and fixing membership in $A$ for
all other relevant strings of other length, we will see that these
circuits are quasipolynomial size in it's number of inputs $2^n$. 

The way this works is as follows. At the top of the circuit is an OR gate,
representing the leading $\exists$ quantifier. This gate has fanin
$O(2^{n^k})$, where $k$ is the polynomial bound of the PH-expression;
one wire for every possible string of length $\leq n^k$. Each of these
wires is the lone output of an AND gate, representing the first
$\forall$, which itself has fanin $O(2^{n^k})$ for the same reason. We
continue like this until reaching a depth equal to the number of
alternating quantifiers $l$. Note at this point that with $y$ fixed,
each bottom level gate corresponds to a particular selection of
$x_1,\ldots,x_n$; all actual inputs to the PH-expression are fixed by
this point. As mentioned, what we see as varying is not the inputs to
the expression but rather the strings which may or may not be in the
oracle $A$, and in particular those which are exactly length $n$ (we
will assume that all strings of relevant length not equal to $n$ have
been fixed as in or out of $A$, and therefore these wires can be seen
as literals). In this way, the value of $R^A(x_1,\ldots,x_l,y)$ for
any selection of inputs can be seen a function of some 
subset of the $2^n$ strings of length $n$ which may or may not
be in the oracle. Each bottom level gate representing the final
quantifier can therefore be seen as being fed input from a DNF circuit
in which each bottom level AND gate represents a particular selection
of strings to explicitly have or not have in $A$ so as to ensure a
positive result. Finally we note that the fanin of these parent OR
gates can all be assumed the same as more or less the same as the
others: $O(2^{n^k})$. The reason for this is that when a machine makes
a query, it's computation goes in one of exactly two ``directions,'' and
at most polynomially many such choices are made. Thus one can imagine
a binary tree of polynomial depth, where some number of these results
in acceptance. Since there are at most $O(2^{n^k})$ many accepting
``paths,'' there are also at most $2^{n^k}$ many combinations of strings
which need to be explicitly specified to ensure acceptance of the
machine. We are therefore left with a depth $l+1$ circuit which takes
$m=2^n$ many inputs, with size $O(2^{l\log^k(m)})$, which is
equivalent to the polynomial expression. The quasipolynomial size is
sufficient to be sure that the circuits in question cannot decide
parity, and this ensures the existence of a set of strings which can
be added to the oracle such that the parity of it doesn't match the
circuit's output.

Specifically, for a given oracle $A$ we consider the language
\[ 
L_A = \{1^j: \textrm{ there is an even number of strings of length $j$ in $A$}\}. 
\]
Note that $L_A$ is in $\bm{PSPACE}^A$ for any $A$, 
since one can simply keep a tally using linear space
while repeatedly querying every length $n$ string. We will 
use the above circuit construction and Theorem~\ref{lowerBound}
diagonalize out of $\bm{PH}^A$ while staying inside of $\bm{PSPACE}^A$
when $b\in 2^\omega$ is in $H$. We use an argument similar to 
that of the previous theorems.

\begin{theorem} \label{Shard}
$\bm{S}$ is $\bP_2^0$ hard.
\end{theorem}

\begin{proof}
As before, fix an $\bm{EXP}$-complete language $E$. We will
define a continuous mapping $f\colon 2^\omega \to 2^\omega$
such that a string $b$ is in $H$ iff $\bm{PH}^{f(b)} \neq
\bm{PSPACE}^{f(b)}$, where $H$ is the collection of strings
with infinitely many $1$'s just as before. The proof will be
similar in strategy to what we've done twice before: perform a
delayed diagonalization in which we only do the next stage
upon encountering a $1$ for the fixed bit string $b$ which we
are mapping to an oracle $A$. Let $g(n)$ be an increasing
function which grows fast enough to ensure that the circuit
representing a PH-expression on $1^{g(n)}$ can never decide
parity by \ref{lowerBound}, and that no strings of length at
least $g(n)$ could possibly appear as relevant to circuits
constructed in previous stages. For the first condition
it will suffice to have 
$2^{\log(n) g(n)^{\log(n)}} <2^{\frac{1}{100} 2^{g(n)/\log(n)}}$, which
will hold for any increasing $g$ for all large enough $n$. 
For the second condition it suffices to have 
$g(n+1)> g(n)^{\log(n)}$, as in Theorem~\ref{Qhard}.
So, $g(n)=2^{2^{2^n}}$ suffices.

We describe a stage $n$
of the construction. First of all, regardless of the
value of $b(n)$, we will always add to $A$ all strings from $E$
of lengths $k$ with $g(n)\leq k <g(n+1)$.
When $b(n)=0$, this is all
we do. Now suppose $b(n)=1$. In this case, we consider the
$d(n)^{th}$ PH expression $\phi$ as in Equation~\ref{phexp}, 
where again $d(n)$ is the number of $m \leq n$ with $b(m)=1$.
Say that $k \in \omega$ is the
polynomial bound of the expression, and $l$ the number of
quantifier alternations. Create as described above a circuit
representing the PH-expression, with the input $y$ fixed as
$1^{g(n)}$, which has depth $l+1$ and size $O(2^{l\log^k(m)})$
where $m := 2^{g(n)}$, and whose inputs correspond to strings
of length $g(m)$ which we may or may not decide to add to the
oracle. We are promising as before to not add any strings of
length not equal to $g(n)$ at stage $n$ of the
construction. By this commitment, along with a commitment to
not add any strings of lengths between $g(n)$ and $g(n+1)-1$
besides those from $E$, we can assume that strings of length
$g(n)$ are the only non-literal inputs to the circuit. By
construction we can be sure that this circuit fails to decide
the parity function for strings of length $g(n)$. We use here 
the fat that the PH expression $\phi$ will be considered at infiitely 
many stages $n$, and that for large enough such $n$ we will have that 
$2^{l\log^k(m)} < 2^{\log(n) \log^{\log(n)} (m)}=2^{\log(n) g(n)^{\log(n)}}$
which is less than $2^{\frac{1}{100} 2^{g(n)/\log(n)}}$, which is 
below the bound of Theorem~\ref{lowerBound}.
This means that there must exist an assignment to the inputs such that
the output of the circuit cannot be the parity of the
assignment, i.e., there exists a collection of strings of
length $g(n)$ which we can add to $A_{n-1}$ so that the $d(n)^{th}$
PH-expression cannot be deciding the test language
\[ 
L_A = \{1^n: \textrm{ $A$ has an even number of strings of length $n$}\} 
\]

This completes the description of $A$. Define the function $f$ by 
setting $f(b)=A$, where $A=\bigcup_n A_n$ is the oracle constructed
from $b$.  It is clear that if $b$ has an infinite
number of $1$'s, then we will ensure that every PH-expression
disagrees with the language $L_{f(b)}$ on at least one input, and
therefore $L_A \notin \bm{PH}^A$. On the other hand, if $b$ has only
finitely many $1$'s, then $A$ will be equal to $E$ up to a finite
difference, and thus $\bm{PH}^{f(b)} = \bm{PSPACE}^{f(b)}$. Thus we
have confirmed that $b\in H$ iff 
$\bm{PSPACE}^{f(b)}=\bm{PH}^{f(b)}$. 
As before, continuity is clear
from the observation that deciding if a string is in $f(b)$ only
requires a finite initial segment of $b$, which is clear from the
construction.
\end{proof}

\begin{corollary}
$\bm{S}$ is $\bP^0_2$-complete.
\end{corollary}

\subsection{The Descriptive Complexity of More Common Classes}
Lastly, we consider the descriptive complexity of common complexity
classes themselves. Since these collections of oracles are countable, 
they cannot be $\Pi^0_2$ complete as in our previous theorems. 
However, it is quite
easy to demonstrate that virtually all of them are $\Sigma^0_2$
and $\bS^0_2$-complete. 
We will consider the class $\bm{P}$ as our prototypical
example. It is easy to see that $\bm{P}$ is naturally a $\Sigma^0_2$
set. Simply note that a language is in $\bm{P}$ iff there
exists a Turing machine exists which agrees with the language on
all inputs. So, $A \in P$ iff $\exists i \ \forall n\ [(n\in A)
\leftrightarrow M_i(n)=1)]$, which shows $\bm{P} \in \bS^0_2$.
The same simple computations works for the other complexity classes.

It is also fairly
simple to show $\bS^0_2$ hardness for $\bm{P}$ and the other classes. 
For example, we prove the following. 

\begin{theorem}
$\bm{P}$ is $\bS^0_2$-complete. 
\end{theorem}

\begin{proof}
As the set $H$ of bit strings which are $1$ infinitely often is $\bP^0_2$-complete, 
the set $2^\omega \sm H$ of bit strings which are $1$ finitely often is $\bS^0_2$-complete. 
We define a reduction $f$ of $2^\omega \sm H$
to $\bm{P}$. For any $b \in 2^\omega$ we define 
$f(b)=L$ as follows. 
Let $d(i)$ be the number of $1$s which have occurred 
by the $i^{th}$ bit.
We create $L$ in stages. 
At stage $n$, if $b(n)=0$ we set $L_n=L_{n-1}$. 
If $b(n)=1$, we run the $d(n)$th polynomial time machine on input $1^n$  for time $n^{\log(n)}$.
If that computation terminates with a rejection, then we let $L_n=L_{n-1}
\cup \{ 1^n\}$, and otherwise set $L_n=L_{n-1}$. 
If $b$ has 
finitely many $1$s, then we will only have added a finite number of strings to $L$ 
so clearly $L\in \bm{P}$. 
On the other hand if $b(n)=1$ for infinitely many $n$, 
then we have diagonalized against all polynomial time machines (recall
that our enumeration of polynomial time machines repeats each machine infinitely often,
and $\log(n)$ will eventually be larger than any fixed $k$).
Thus the language $L$ created is in $\bm{P}$ iff $b\in 2^\omega \sm H$.

\end{proof}

It is clear that the argument just supplied for $\bm{P}$ works for other  complexity classes. 
We have the following, for example. 

\begin{theorem}
$\bm{NP}$, $\bm{EXP}$, $\cnp$, $\bm{BQP}$, $\bm{PH}$, $\bm{L}$, $\bm{NL}$, $\bm{PSPACE}$ are all
$\bS^0_2$-complete.
\end{theorem}

\section{Ramsey Theory}\label{section:Ramsey}

The notion of {\em Ramsey large} is, aside from measure and category, another important 
measure of size for sets $S\subseteq [\omega]^\omega$. Recall we are identifying infinite subsets 
of $\omega$ with increasing functions $f \colon \omega \to \omega$. 
We let $S^c$ denote the complement $[\omega]^\omega \sm S$. 
One way to  define this notion is through the {\em Ellentuck topology}.  Letr $a$ be a finite subset 
of $\omega$ and $A\subseteq \omega\sm s$ an infinite set. Then a basic open set 
in the Ellentuck topology is a set of the form 
\[ [a,A] = \{S \in [\omega]^{\omega}\colon S\sm A=a \wedge S\sm a \subseteq A\} \]
If we choose to view the infinite subsets of $\omega$ as increasing functions,
then 
\[ [a,A] = \{ f \in [\omega]^{\omega}\colon
\forall i< |a| (a(i)=f(i)) \wedge \forall i \geq  |a| (f(i) \in A)\} \]

The Ellentuck topology is closely related to {\em Mathias forcing} in set theory,
where the objects $[a,A]$ as above are the elements of the forcing partial order. 
The Ellentuck topology/ Mathias forcing arises naturally in the study of 
partition properties on $\omega$. Recall the classical Ramsey theorem
says that for any partition $P\colon \omega^m \to k$, where $m,k \in \omega$,
there is an infinite set $H\subseteq \omega$ which is homogeneous 
for the partition. That is, $P\res [H]^\omega$ is constant. With the axiom of choice, 
no infinite cardinal, including $\omega$, can have an infinite exponent 
partition property, but if we restrict to reasonable definable 
sets $S \subseteq [\omega]^\omega$, then such partition relations can hold. 
The Galvin-Prikry theorem, stated next, says that Borel partition of $[\omega^\omega]$
have homogeneous sets. 

\begin{theorem}[Galvin-Prikry]
Let $[\omega]^{\omega} = P_0 \cup P_1 \cup \ldots \cup P_{k-1}$ be a partition of $[\omega]^{\omega}$, 
with each $P_i$ is Borel (in the usual topology on the Baire space).
Then there exists an infinite $H \subseteq \omega$ which is homogeneous for 
the partition. That is, there is an $i<k$ such that $[H]^\omega \subseteq P_i$. 
\end{theorem}

Silver extended this result to include all $\bS^1_1$ partitions. Assuming Woodin's
axiom $\ad^+$, every partition $P$ of $[\omega]^\omega$ into finitely many 
pieces has a homogeneuous set. For this paper, however, it will be enough to consider Borel partitions.

A set $S \subseteq [\omega]^\omega$ is called {\em Ramsey} if there is an $H\in [\omega]^\omega$
such that $[H]^\omega \subseteq B$, or $[H]^\omega \subseteq [\omega]^\omega]\sm B$. In other words,
$S$, viewed as a partition of $[\omega]^\omega$ into two pieces, has a homogeneous set. 
Thus, the Galvin-Prikry theorem asserts every Borel set in $[\omega]^\omega$ is Ramsey. 
A somewhat stronger notion is that of $S$ being {\em completely Ramsey}. 
This means that for every basic open set $[a,A]$ we have an infinite $H\subseteq A$
such that $[a,H]\subseteq S$ or $[a,H]\subseteq S^c$. 
In fact, the Galvin-Prikry theorem asserts that every Borel set $S\subseteq [\omega]^\omega$ is completely 
Ramsey.

A basic fact about the Ellentuck topology is that a set $S \subseteq [\omega]^\omega$ is 
nowhere dense iff for every basic open set $[a,A]$ there is a 
$B\in [A]^\omega$ such that $[a,B] \subseteq S^c$. This differs from the definition of nowhere dense in that 
we do not need to change the stem $a$, but just need to thin out the set $A$. 
Also a set $S \subset [\omega]^\omega$ is meager iff it is nowhere dense. 
Thus, a set $S$ is meager iff for every $[a,A]$ there is a $B\in [A]^\omega$ such that 
$[a,B]\subseteq S^c$. The meager sets, as usual, form a $\sigma$-ideal on $[\omega]^\omega$.

Recall that a set $S$ in a topological space has the Baire property iff there is an open set $U$
such that $S\triangle U$ is meager. We have the following characterization of the completely Ramsey 
sets.

\begin{theorem}[Ellentuck]
A set $B \subseteq [\omega]^{\omega}$ is completely Ramsey iff it has the Baire property in the 
Ellentuck topology. 
\end{theorem}

We say a set $S\subseteq [\omega]^\omega$ is Ramsey small (or Ramsey measure $0$) 
iff it is meager in the Ellentuck topology,
and call it Ramsey large (or Ramsey measure $1$) if it is comeager in this topology. A set which is not Ramsey 
small is called Ramsey positive. Note that a set $S$ with the Baire property in the 
Ellentuck topology is Ramsey positive iff it contains a non-empty basic open set $[a,A]$.

In this section we appply the notion of Ramsey largeness to the study of oracles separating 
complexity classes. We have two goals in mind in proving these results. The first is to generalize the 
known results concerning the largeness of these collections 
with respect to measure and category. The second is to provide a possible plausible 
version of the random oracle hypothesis, which we discuss below.

We first show that the set $\bm{O'}$ of oracles (viewed as infinite subsets of $\omega$) 
$K$ such that $\np^K\neq \cnp^K$ is Ramsey positive. Recall we are identifying oracles $K$
with subsets of $\omega$ by our fixed enumeration of the finite strings.

\begin{theorem}\label{Ohomo}
The set $\bm{O'}$ is Ramsey positive. That is, 
there exists a homogeneous  $H \subseteq \omega$ such that for all infinite oracles $K$ 
consisting of strings from $H$, $\np^K \neq \cnp^K$.  
\end{theorem}

\begin{proof}
We established in \S\ref{section:descriptiveComplexity} that $\bm{O'}$ is a Borel set.
By the Galvin-Prikrey theorem, this means that the set is completely Ramsey. To show 
$\bm{O'}$ is Ramsey positive it suffices to get an $A\in [\omega]^\omega$ 
such that $[A]^\omega \subseteq \bm{O'}$. By Galvin-Prikry it suffices to 
construct an $A$ such that for every infinite $B\subseteq A$ there is an infinite $C\subseteq B$
such that $\bm{NP}^C \neq \cnp^C$. We proceed to construct such a set $A$.

We fix a sufficiently fast growing function $f \colon \omega \to \omega$, say 
with $f(n+1) > f(n)^{\log(f(n))}$, and we may also assume that 
$n 2^{n-1} f(n)^{\log(f(n))} < 2^{f(n)}$as the function $2^x$ grows faster than $x^{\log(x)}$. 
Our set $A$ will have exactly $1$ string of length $f(n)$ for all $n$. 
We construct $A$ in stages, and refer to the partially constructed set at stage $n$ by $A_n$.
At stage $n$, we run for each $m \leq n$ and each subset $S\subseteq A_{n-1}$, 
the $m$th $\bm{NP}$ machine on the partial oracle $S$ 
with input $1^{f(n)}$ for  $f(n)^{\log(f(n))}$ many steps. For each such pair $(m,S)$,
if this run of the $m$th $\bm{NP}$ machine results in an acceptance, then there is a valid path in the 
configuration graph for the machine of length at most $f(n)^{\log(f(n))}$ which terminates in
an acceptance state. This particular computation path queries at most 
$f(n)^{\log(f(n))}$ many strings of length $f(n)$. We fix a set $E(m,s)$ of strings of length $f(n)$
and $|E(m,S)| \leq f(n)^{\log(f(n))}$ which contains this set of queries. Note that in this case 
that if $A'$ agrees with $A_{n-1}$ on all strings of length $< f(n+1)$ except 
for strings of length $f(n)$, and if $A' \cap E(m,S)=\emptyset$, then the $m$th $\bm{NP}$
machine run with the oracle $A'$ (with input $1^{f(n)}$ for 
$f(n)^{\log(f(n))}$ many steps) will also terminate in acceptance.
Let $E_n =\cup \{ E(m,S)\colon m \leq n \wedge S\subseteq A_{n-1}\}$. 
We have $|E_n|\leq n 2^{n-1} f(n)^{\log(f(n))} <2^{f(n)}$. Thus there is a string $s_n$ 
of length $f(n)$ with $s_n \notin E_n$. We let $A_n= A_{n-1} \cup \{ s_n\}$. 
We let $A=\bigcup_n A_n$. This completes the definition of the set $A$.

Suppose now $B\in [A]^\omega$ is an infinite subset of $A$. We show that there is a $C\in [B]^\omega$
such that $\np^C\neq \cnp^C$. Let $g(n)$ be the $n$th element of $B$. Clearly $g(n)\geq f(n)$. 
We construct $C\subseteq B$ by a diagonalization argument similar to our previous proofs. 
We construct $C$ in stages and let $C_i$ be the set constructed at stage $i$. 
We will have $C_i \subseteq B\cap [0,g(i)]$. 
Assume $C_{n-1}$ has been defined. We run the $n$th $\bm{NP}$ machine on the input 
$1^{g(n)}$ for $g(n)^{\log(g(n))}$ steps relative to the oracle $C_{n-1}$. If this computations halts 
in an acceptance, that is there is an accepting path in the configuration graph of the 
non-deterministic computations, then we let $C_n=C_{n-1}\cup \{ s_n\}$, where 
$s_n$ is the unique string of length $g(n)$ in $B$. Otherwise we set $C_n=C_{n-1}$.

It remains to show that $\np^C \neq \cnp^C$. 
The language witnessing this difference will be the standard Solovay language $L_C$
(the set of strings $1^k$ such that there is a string of length $k$ in $C$), 
which as always is clearly in $\np^C$.
Suppose it were also in $\cnp^C$. Then $\co L_C$ 
(the complement of $L_C$) would be in $\bm{NP}^C$. That is, there would exist polynomial time 
(say time $n^k$ for a fixed $k \in \omega$) relativized nondeterministic machine $N^C$ which accepts it.
Pick an $m$ sufficiently large so that $g(m)^{\log(g(m))} > g(m)^k$, and where the $m$th non-deterministic machine 
$N^C_m$ is $N$.  Consider the operation of this machine, run for   $g(m)^{\log(g(m))}$ steps, 
on input $1^{g(m)}$ relative to the oracle $C_{n-1}$. Suppose first this computations halts with an acceptance. 
Then $C_n=C_{n-1}\cup \{ s_n\}$, and $s_n$ is not queried in some accepting path for the nondeterministic 
computation. Thus $N_m^{C_n}(1^{g(n)})$ also halts with an acceptance. Since $g(n+1) \geq f(n+1)
\geq f(n)^{\log(f(n))}$, it then follows that $N_m^C(1^{g(n)})$ also halts with an acceptance. 
Since there is a string of length $g(n)$ in $C$, namely $s_n$, the machine $N_m^C$
has made an error on the input $1^{g(m)}$. Suppose next that $N_m^{C_{n-1}}(1^{g(n)})$ 
does not halt with an acceptance when run for $g(m)^{\log(g(m))}$ many steps. So, $C_{n}=C_{n-1}$. 
So, $N_m^{C_n}(1^{g(m)})$ also does not halt with an acceptance, and 
likewise $N_m^C(1^{g(m)})$ does not halt with an acceptance when run this many steps. Since 
$1^{g(m)} \notin L_c$ in this case, the machine $N_m^C$ has made an error on the input $1^{g(m)}$
in this case as well. We have confirmed that no polynomial time relativized 
nondeterministic Turing machine can possibly decide $\co L_C$, completing the proof. 
\end{proof}

Since $\bm{O'}\subseteq \bm{O}$, we immediately have the following.

\begin{corollary} \label{cor:Ohomo}
Both $\bm{O}$ and $\bm{O'}$ are Ramsey positive. 
\end{corollary}

A similar argument applied on top of the argument from Theorem~\ref{Qhard} 
will separate $\np$ from $\bm{BQP}$ according to the following theorem.  

\begin{theorem}\label{Qhomo}
The set $\bm{Q}$ is Ramsey positive. That is, there exists a homogeneous set $H\in [\omega]^\omega$ 
such that for any $C \in [H]^\omega$ we have $\np^C \nsubseteq \bqp^C$. 
\end{theorem}

\begin{proof}

As in the proof of Theorem~\ref{Ohomo} we construct an infinite set $A\subseteq \omega$ (again identified with
a set of strings) such that for any $B\in [A]^\omega$ there is a $C\in [B]^\omega$ such that 
$\np^C \nsubseteq \bqp^C$. We fix a fast growing $f \colon \omega \to \omega$ satisfying 
$f(n+1) > f(n)^{\log(f(n))}$ and $n2^n 1568 f(n)^{2\log(f(n))} < 2^{f(n)}$ for all $n$. 
We again define $A$ in stages, and let $A_n$ be the set defined at stage $n$. We will 
get $A_n$ from $A_{n-1}$ by adding at most one string of length $f(n)$. Let $M_m$ enumerate all of the 
quantum Turing machines, with each machine repeated infinitely often. At stage $n$,
suppose $A_{n-1}$ has been defined. 
For each $m\leq n$ and each $S\subseteq A_{n-1}$, consider running the $m$th quantum machine 
$M_m$ on the input $1^{f(n)}$ relative to the oracle $S$ for $T=f(n)^{\log(f(n))}$ many steps. 
Let $\ket{\phi}^{m,S}_t$ for $t \leq T$ denote the superposition at time $t$. If $y$ is a string of length $f(n)$, 
let $S_y$ denote the oracle $S_y=S\cup \{ y\}$. 
Let $E(m,S)$ denote the set of strings $y$ of length $f(n)$ 
such that $\sum_{i=0}^{T-1}q_y(\ket{\phi}_i^{m,S}) < \frac{1}{1568T}$, where as in Theorem~\ref{Qhard}
$q_y(\ket{\phi}_i^{m,S})$ the the sum of the squared magnitudes of the amplitudes of configurations in the 
superposition $\ket{\phi}_i^{m,S})$ corresponding to a query state in which the string $y$ is being queried. 
In the proof of Theorem~\ref{Qhard} we showed that $|E(m,S)|\leq 1568 T^2$. 
So we have:
\[|\bigcup_{m,S} E(m,S)| \leq n 2^n 1568 T^2= n 2^n 1568 f(n)^{2\log(f(n))} <2^{f(n)}.
\]
Let $A_n=A_{n-1} \cup \{y_n\}$ where $y_n$ is a string of length $f(n)$ not in 
$E_n=\bigcup_{m,S} E(m,S)$. This completes the definition of the set $A=\bigcup_n A_n$.

Suppose now $B\in [A]^\omega$. We construct a $C\in [B]^\omega$ such that 
$\np^C\nsubseteq \bqp^C$. In fact, we will have that the Solovay language 
\[
L_C=\{ 1^k \colon \,\text{there is a string of length $k$ in $C$} \}
\]
is not in $\bqp^C$.

Let $g(n)$ be the $n$th element of $B$, so $g(n)\geq f(n)$. 
We construct $C$ in stages as before, and at stage $n$ we get $C_n$ from $C_{n-1}$
by adding at most one string of length $g(n)$, which would have to be the unique string in $B$ of length $g(n)$. 
Suppose we are at stage $n$ and $C_{n-1}$ has been defined. Let $S=B\cap [0,g(n-1)]$. Consider the quantum machine
$M_n$ run on input $1^{g(n)}$ relative to the oracle $S$ for $T=g(n)^{\log(g(n))}$ many steps. 
Let $\ket{\phi}_t$ be the superposition of this computation at time $t \leq T$. 
Suppose first that there is a time $t \leq T$ such that  $\ket{\phi}_t$ has a greater than $\frac{1}{2}$ 
probability of being in a configuration which is a halting state which is a rejecting state. 
In this case we ste $C_{n}= C_{n-1}\cup \{ s_n\}$ where $s_n$ is the unique string in $B$ of length $g(n)$. 
In all other cases we set $C_n=C_{n-1}$. The completes the definition of $C$.

Suppose towards a contradiction that the Solovay language $L_C$ were in $\bqp^C$. 
Let $M$ be a quantum Turing machine, say with run time bounded by $n^k$ for input size $n$,
which decides the language $L_C$. Fix $m$ such that $M_m=M$ and 
$m$ is large enough so that $\log(m)>k$. Consider stage $m$ in the construction of $C$. 
Let $\ket{\phi'}_t$ for $t \leq T$ be the superposition at time $t$ for the
quantum computation $M_m^{C}(1^{g(m)})$, and let 
$\ket{\phi^{s_m}}_t$ denote the superposition for the computation 
$M_m^{C_{n-1}\cup \{ s_m\}}(1^{g(m)})$ where $s_nm\in A$ is the unique string in $A$ of length $g(m)$. 
Since $g(m+1)\gg g(m)$ we have (since $C_m$ is either $C_{m-1}$ or $C_{m-1}\cup \{ s_m\}$)
that $\ket{\phi'}_t$ is either $\ket{\phi}_t$ or $\ket{\phi^{s_m}}_t$.

Suppose first that $M_m^C(1^{g(m)})$ accepts, that is, there is a time $t_0 \leq T$ such that 
$\ket{\phi'}_{t_0}$ is in a superposition of halting configurations and there is a $>\frac{2}{3}$
probabilty that these halting states are accepting. Furthermore, in this case we must have that 
for all $t'<t_0$ that $\ket{\phi'}_{t'}$ is in a superposition of configurations, none of which is in 
a halting state (from the definition of a well-formed quantum Turing machine halting, 
see \cite{quantumComplexityTheory}). 
For such $t'$, $\ket{\phi}_{t'}$ has a $<\frac{1}{2}$ probability of being in a halting state 
since from  the choice of $s_m$ in the construction of $A$ and 
Lemma~\ref{oracleAlter} we have that $| \ket{\phi}_{t'} -\ket{\phi'}_{t'}|<\frac{1}{28}$
and so from Lemma~\ref{distDist} we have that the difference between the probability of 
$\ket{\phi}_{t'}$ being in a halting state and $\ket{\phi'}_{t'}$ being in a halting 
state is at most $\frac{1}{7}$. Thus, the time $t$ used in the definition of $C_m$ is equal to $t_0$. 
Also, the probability that $\ket{\phi}_{t_0}$ is in a halting state which rejects the input 
is at most $\frac{1}{3}+\frac{1}{7}<\frac{1}{2}$. So by definition of $C_m$ we have 
in this case that $C_m=C_{m-1}$, and so $1^{g(m)} \notin L_C$. Thus, the machine $M=M_m$
has made an error.

Suppose next that $M_m^C(1^{g(m)})$ rejects, and let $t_0$ again be the corresponding time.
As in the previous paragraph, the time $t$ as in the definition of $C_m$ is equal to $t_0$. 
Also as in the previous paragraph, the probability that $\ket{\phi}_{t_0}$ is a halting 
configuration which rejects the input is at least $\frac{2}{3}-\frac{1}{7}>\frac{1}{2}$. 
So by definition of $C_m$ we have $C_m=C_{m-1}\cup \{ s_m\}$, so $1^{g(m)} \in L_A$. 
Thus again the machine $M$ has made an error on the input $1^{g(m)}$. 

In both cases we conclude that it cannot be the case that $M$ actually decides $L_C$. 
Thus it cannot be the case that $L_C \in \bqp^C$, and so $\np^C\nsubseteq \bqp^C$. 

\end{proof}

Surprisingly, the analogous statement for $\bm{S}$, that $\bm{S}$ is Ramsey positive, 
appears to be stronger than the other two. We discuss this further in \S\ref{section:hypothesis} below.

\subsection{The Ramsey Oracle Hypothesis}\label{section:hypothesis}

In this section we present results which appear to suggest that the Ramsey version of the random oracle hypothesis
may hold. That is, if two complexity classes are equal when relativized to a Ramsey large set of oracles,
then the two unrelativized classes are equal. Note in this regard that Theorems~\ref{Ohomo} and \ref{Qhomo}
only show that for the classes considered there that the set of separating oracles is Ramsey positive. 
We will show for the classes considered in those results that if the collection of separating 
oracles is actually Ramsey large, then the corresponding unrelativized classes are equal. 
For the set $\bm{S}$ of oracles separating $\bm{PSPACE}$ and $\bm{PH}$, we will show the stronger
result that if $\bm{S}$ is just Ramsey positive, then $\bm{PSPACE}\neq \bm{PH}$.

The following result demonstrates that showing that $\bm{O}$ is Ramsey large is at 
least as difficult as the $\bm{P}= \bm{NP}$ problem itself.

\begin{theorem} \label{orl}
If the set $\bm{O}$ of oracles separating $\bm{P}$ and $\np$ is Ramsey large then
$\bm{P}\neq \np$.
\end{theorem}

\begin{proof}
Assume $\bm{P}=\np$ and we produce an infinite set $A\subseteq \omega$ (which we are identifying
with a set of strings) such that for any $B\in [A]^\omega$ we have $\bm{P}^B=\np^B$, which
contradicts $\bm{O}$ being Ramsey large. Let $A=\{ 1^n \colon n \in \omega\}$, or more generally
we can take $A=\{ 1^{f(n)} \colon n \in \omega\}$ where $f \colon \omega \to \omega$ 
is any increasing function $f$. Suppose $B \in [A]^\omega$. We show $\np^B=\bm{P}^B$. 
Let $N^B$ be a relativized nondeterministic Turing machine which runs in polynomial time,
say bounded by $n^k$ where $n$ is the input length, and which decides a language $L$. We describe a 
deterministic machine $M$ such that $M^B$ runs in polynomial time and decides $L$. 
First, there is a deterministic machine $M_1^B$ which on an input of length $n$ 
computes a ``look-up table'' for $B$, that is, it records for all strings $s$ of the form $1^m$, for 
$m \leq n^k$, whether or not $s \in B$. Of course, for strings $s$ not of this form, we know
that $s \notin B$. Second, there is an unrelativized nondeterministic machine 
$N_2$ which takes as input a string $s$ of length $n$ and a string $t$ of length $\leq n^k$
and simulates the machine $N$ in the following sense. Whenever a computation path of the 
nondeterministic machine $N$ would query the oracle on a string $u$, the machine $N_2$ 
proceeds as if the answer is ``no'' if $u$ is not of the form $1^m$, and if $u$ 
is of this form then $N_2$ uses the answer ``yes'' iff $u$ has length $\leq n^k$ 
and $t(m)=1$. Thus, the string $t$ is viewed as coding the relevant part of an oracle $C$. 
By the assumption that $\bm{P}=\np$, there is a polynomial-time deterministic machine
$M_2$ which accepts the same language as $N_2$ (where the language is now viewed as a set of pairs 
$(s,t)$.)

The deterministic machine $M^B$ on $s$ of length $n$ first runs $M_1^B$ to generate the output 
string $t$ which codes the strings $1^m \in B$ for $m \leq n^k$. Then $M^B$ 
runs the machine $M_2$ of the input $(s,t)$. This will produce the same answer as $N^B(s)$. 
\end{proof}

At first glance one might expect that showing a similar result for $\bm{Q}$ to be more complicated 
since the analogous hypothesis (that $\bm{NP} \subseteq \bm{BQP}$) seems weaker. This however poses no problems.

\begin{theorem} \label{qrl}
If $\bm{Q}$ is Ramsey large, then $\bm{NP} \nsubseteq \bm{BQP}$
\end{theorem}

\begin{proof}
Consider again the set $A = \{1^n: n \in \omega\}$, and suppose that $\bm{NP} \subseteq \bm{BQP}$. 
Just as in Theorem~\ref{orl}, we show that for any $B \subseteq A$ that $\bm{NP}^B \subseteq \bm{BQP}^B$
which contradicts the hypotheis that $\bm{Q}$ is Ramsey large. 
Let $N^B$ be a relativized nondeterministic Turing machine which runs in polynomial time,
say bounded by $n^k$ where $n$ is the input length.
As in Theorem~\ref{orl},
there is a relativized deterministic machine $M_1^B$ which on input $s$ records whether or not 
$1^m \in B$ for all $m \leq |s|^k$. We can also view $M_1^B$ as a relativized polynomial time
quantum machine, since $\bm{P}\subseteq \bqp$. Let $N_2$ be the unrelativized nondeterministic 
machine of Theorem~\ref{orl}. Since we are assuming $\np \subseteq \bqp$, there 
is an unrelativized polynomial time quantum machine $Q$ which accepts the same language as 
$N_2$. As before, by concatenating the machines $M_2^B$ and $Q$ we obtain a 
$\bqp^B$ machine which accepts the same language as $N^B$. 

\end{proof}

As we mentioned earlier, the analogous statement for $\bm{S}$ is stronger than the other two. 
Simply assuming the set $\bm{S}$ is Ramsey positive is enough to prove $\bm{PH} \neq \bm{PSPACE}$. 
To highlight the difference we make the following definition.

\begin{definition} \label{def:tame}
An oracle $O\subseteq 2^{<\omega}$ is {\em tame} if it has at most one string of any 
given length. We say $O$ is {\em very tame} if there is a polynomial time computable tame
oracle $A$ such that $O\subseteq A$. 
\end{definition}

In Theorems~\ref{orl} and \ref{qrl}, the relevant property of the oracle $\{ 1^n \colon n \in \omega\}$
is that is very tame. The point is that relativized $\bm{NP}^B$ machines are capable of creating lookup tables 
for tame oracles, whereas relativized $\bm{P}$ machines are presumably only capable of doing the same 
for very tame oracles. So, for the class $\bm{S}$ of oracles separating $\bm{PSPACE}$ from $\bm{PH}$
we have the following result.

\begin{theorem} \label{Shomo}
If $\bm{S}$ is Ramsey positive, then $\bm{PSPACE} \nsubseteq \bm{PH}$.
\end{theorem}

\begin{proof}
Suppose that $\bm{PSPACE} = \bm{PH}$ and let $A$ be some oracle. It suffices to that there 
is an oracle $B\subseteq A$ such that $\bm{PSPACE}^B = \bm{PH}^B$, as this contradicts 
$\bm{S}$ being Ramsey positive. Let $B\subseteq $ be tame, and we show  $\bm{PSPACE}^B = \bm{PH}^B$.
Consider a language $L \in \bm{PSPACE}^B$, 
decided by the machine $M^B$ in space $n^k$. 
We can have a relativized $\bm{NP}^B$ machine which on input $s$ creates a  
lookup table $T_{|s|}$ for strings of length $\leq |s|^k$ by guessing a single string of each 
length up to $|s|^k$, and querying $B$ to check membership. This is possible since $B$ is tame. 
There is a nonrelativized $\bm{PSPACE}$ machine $M'$ which behaved identically
to $M^B$ provided it receives the correct lookup table, that is, 
$M'(T_{|s|},s) = M^B(s)$ for all strings $s$. 
From the hypothesis 
$\bm{PSPACE} = \bm{PH}$ we can get a nonrelativized,  $\bm{PH}$ expression $\varphi$
which on inputs $s$ and the lookup table $T$ will be equivalent to $M'$,
that is $M'(s,T)$ accepts iff $\varphi(s,T)$. The ``concatenation'' of the $\np$ machine 
$N^B$ and the $\bm{PH}$ expression $\varphi$ can be expressed by a relativized $\bm{PH}$ 
expression $\psi$ in which a extra block of existential quantifiers is used
to quantify over the existence of the lookup table $T$ as produced by $N^B$. 
Then for all strings $s$ we have $\psi^B(s)$ iff $M^B(s)$ accepts, which
shows $L\in \bm{PH}^B$. 
\end{proof}

From our results we see that Ramsey largeness is a very strong condition to have on 
sets of oracle separating two classes.
To reflect on why this is, consider an arbitrary oracle $A$. 
One might show that two classes relativized to $A$ are different with probability $1$, but this says nothing about the 
kinds of the problems doing the separating. To have that the relevant separating class of oracles be
Ramsey large is to say that one can always distill out of an arbitrary racle (by passing to a subset)
something which separates them, 
and that includes starting from oracles which are already very restricted, for example, tame or very tame oracles. 
To have a separating class of oracles be Ramsey positive is to 
have that the classes are evidently ``different enough'' that a tame oracle, despite its limitations, is 
still powerful enough to create a difference. Similarly, 
if the separating class of oracles is Ramsey large, then even a very tame oracle can separate
the classes.

In the case of $\bm{PH}$ vs $\bm{PSPACE}$, 
the classic separation involving exploiting the lack of small circuits deciding parity requires one to 
explicitly add multiple strings of the same length at every step. From our results then it seems 
likely that if one were to find a diagonalization method in which the oracle constructed is specifically tame, 
then that would lead to a proof that the classes were unequal, via an analogous argument to 
either Theorem~\ref{orl} or \ref{qrl}.  Moreover the restrictedness of tame oracles and of very tame oracles seems 
closely related to the power of $\bm{NP}$ machines and of $\bm{P}$ machines, respectively. The 
condition of being Ramsey positive, in the case again of $\bm{PH}$ and $\bm{PSPACE}$, might therefore 
be considered in another way: that they are so different that a confirmed easy $\bm{NP}$ computation 
is enough to confirm separation. Likewise by \ref{Qhomo} $\bm{NP}$ is so close to being incomparable 
to $\bm{BQP}$ that a confirmed easy $\bm{P}$ computation could be enough to show full separation. 
Such information seems potentially valuable in the pursuit of linking our wealth of oracle results 
with the world of complexity as it actually exists.

Finally, we mention some questions left open from this work. Concerning
the Ramsey results we ask:

\begin{question}
Do the converses of Theorems~\ref{orl}, \ref{qrl} hold?
\end{question}

In view of Theorems~\ref{orl}, \ref{qrl}, \ref{Shomo} it is natural to ask the following. 

\begin{question}
For which other notions of largeness, and for which other pairs of 
complexity classes does the random oracle hypothesis hold?
\end{question}

\bibliographystyle{plain}

\bibliography{bibliography}
\end{document}